\documentclass[english]{article}
\usepackage[T1]{fontenc}
\usepackage[latin9]{inputenc}
\usepackage{geometry}
\usepackage{babel}
\usepackage{bm}
\usepackage{amsmath}
\usepackage{amsthm}
\usepackage{amssymb}
\usepackage[unicode=true,
 bookmarks=true,bookmarksnumbered=false,bookmarksopen=true,bookmarksopenlevel=2,
 breaklinks=false,pdfborder={0 0 1},backref=false,colorlinks=false]
 {hyperref}

\makeatletter
\theoremstyle{plain}
\newtheorem{thm}{\protect\theoremname}[section]
  \theoremstyle{plain}
  \newtheorem{question}[thm]{\protect\questionname}
  \theoremstyle{definition}
  \newtheorem{defn}[thm]{\protect\definitionname}
  \theoremstyle{plain}
  \newtheorem{lem}[thm]{\protect\lemmaname}
\newcommand{\lyxaddress}[1]{
\par {\raggedright #1
\vspace{1.4em}
\noindent\par}
}

\date{}

\allowdisplaybreaks
\usepackage{bm}

\makeatother

  \providecommand{\definitionname}{Definition}
  \providecommand{\lemmaname}{Lemma}
  \providecommand{\questionname}{Question}
\providecommand{\theoremname}{Theorem}

\begin{document}

\title{\textsc{Noether's Problem for Some Subgroups}\\
\textsc{of $S_{14}$: the Modular Case}}

\author{\textsc{Hang Fu $\cdot$ Ming-chang Kang $\cdot$ Baoshan Wang $\cdot$
Jian Zhou}}
\maketitle
\begin{quote}
\textbf{\small{}Abstract.}{\small{} Let $G$ be a subgroup of $S_{n}$,
the symmetric group of degree $n$. For any field $k$, $G$ acts
naturally on the rational function field $k(x_{1},\cdots,x_{n})$
via $k$-automorphisms defined by $\sigma\cdot x_{i}:=x_{\sigma\cdot i}$
for any $\sigma\in G$ and $1\leq i\leq n$. In this article, we will
show that if $G$ is a solvable transitive subgroup of $S_{14}$ and
$\textup{char}(k)=7$, then the fixed subfield $k(x_{1},\cdots,x_{14})^{G}$
is rational (i.e., purely transcendental) over $k$. In proving the
above theorem, we rely on the Kuniyoshi\textendash Gasch\"utz Theorem
or some ideas in its proof.}{\small \par}

\textbf{\small{}Keywords.}{\small{} Noether's problem $\cdot$ L\"uroth's
problem $\cdot$ Rationality problem $\cdot$ Kuniyoshi's Theorem
$\cdot$ Gasch\"utz's Theorem}{\small \par}

\textbf{\small{}Mathematics Subject Classification.}{\small{} 13A50
$\cdot$ 14E08}{\small \par}
\end{quote}

\section{\label{sec1}Introduction}

Let $k$ be a field, $L$ a finitely generated extension field of
$k$. The field $L$ is called \textit{$k$-rational} (or rational
over $k$) if $L$ is purely transcendental over $k$, i.e., $L$
is $k$-isomorphic to the quotient field of some polynomial ring over
$k$. $L$ is called \textit{stably $k$-rational} if $L(x_{1},\cdots,x_{n})$
is $k$-rational for some $x_{1},\cdots,x_{n}$, which are algebraically
independent over $L$. $L$ is called \textit{$k$-unirational} if
$L$ is $k$-isomorphic to a subfield of some $k$-rational field.
It is obvious that ``$k$-rational'' $\Rightarrow$ ``stably $k$-rational''
$\Rightarrow$ ``$k$-unirational''. L\"uroth's problem asks, under
what situations, the converse is true, i.e., ``$k$-unirational''
$\Rightarrow$ ``$k$-rational''. For a survey of this famous problem
in algebraic geometry, see \cite{MT}.

Let $G$ be a finite group, $k$ a field, $V$ a finite-dimensional
vector space over $k$, and $\rho:G\to\textup{GL}(V)$ a faithful
representation of $G$. Then $G$ acts on the function field $k(V)$
by $k$-automorphisms. Noether's problem asks, under what situations,
the fixed subfield $k(V)^{G}:=\{f\in k(V):\sigma\cdot f=f\textup{ for any }\sigma\in G\}$
is $k$-rational. Noether's problem is a special form of L\"uroth's
problem. It is also related to the inverse Galois problem. See \cite{Sw}
for details.

When $\rho:G\to\textup{GL}(V_{\textup{reg}})$ is the regular representation
of $G$ over $k$, we will write $k(G)$ for the fixed subfield $k(V_{\textup{reg}})^{G}$.
Explicitly, let $G$ act on the rational function field $k(x_{g}:g\in G)$
by $h\cdot x_{g}:=x_{hg}$ for any $g,h\in G$. Then $k(G):=k(x_{g}:g\in G)^{G}$,
the fixed subfield of $k(x_{g}:g\in G)$.

When $G$ is a subgroup of $S_{n}$, the symmetric group of degree
$n$, the permutation representation associated to $G$ induces a
natural action of $G$ on the rational function field $k(x_{1},\cdots,x_{n})$
via $k$-automorphisms defined by $\sigma\cdot x_{i}:=x_{\sigma\cdot i}$
for any $\sigma\in G$ and $1\leq i\leq n$. Noether's problem becomes
the form: whether the fixed subfield $k(x_{1},\cdots,x_{n})^{G}:=\{f\in k(x_{1},\cdots,x_{n}):\sigma\cdot f=f\textup{ for any }\sigma\in G\}$
is $k$-rational. The main purpose of this paper is to study the $k$-rationality
of $k(x_{1},\cdots,x_{n})^{G}$, where $k$ is a field, $G$ is a
transitive subgroup of $S_{n}$, and $n$ is a ``small'' positive
integer.

Recall some previously known results.
\begin{thm}
\label{t1.1} Let $k$ be a field, $S_{n}$ the symmetric group of
degree $n$, and $k(x_{1},\cdots,x_{n})$ the rational function field
of $n$ variables over $k$. Suppose that $G\leq S_{n}$ acts on $k(x_{1},\cdots,x_{n})$
by $k$-automorphisms defined by $\sigma\cdot x_{i}:=x_{\sigma\cdot i}$
for any $\sigma\in G$ and $1\leq i\leq n$.
\begin{enumerate}
\item \textup{\cite[Theorem 1.3]{KW}} If $1\leq n\leq5$, then $k(x_{1},\cdots,x_{n})^{G}$
is $k$-rational for any subgroup $G\leq S_{n}$.
\item \textup{\cite[Theorem 1.2]{KWZ}} If $n=6$ and $G$ is a transitive
subgroup of $S_{6}$ other than $\textup{PSL}_{2}(\mathbb{F}_{5})$,
$\textup{PGL}_{2}(\mathbb{F}_{5})$, or $A_{6}$, then $k(x_{1},\cdots,x_{6})^{G}$
is $k$-rational.
\item \textup{\cite[Theorem 1.4]{KW}} If $n=7$ and $G$ is a transitive
subgroup of $S_{7}$ other than $\textup{PSL}_{2}(\mathbb{F}_{7})$
or $A_{7}$, then $k(x_{1},\cdots,x_{7})^{G}$ is $k$-rational. If
$k\supseteq\mathbb{Q}(\sqrt{-7})$, then $k(x_{1},\cdots,x_{7})^{\textup{PSL}_{2}(\mathbb{F}_{7})}$
is $k$-rational.
\item \textup{\cite[Theorem 1.2]{WZ}} If $n=8$ and $G$ is a solvable
transitive subgroup of $S_{8}$ other than $C_{8}$, then $k(x_{1},\cdots,x_{8})^{G}$
is $k$-rational.
\item \textup{\cite[Theorem 1.5]{KW}} If $n=11$ and $G$ is a solvable
transitive subgroup of $S_{11}$, then $k(x_{1},\cdots,x_{11})^{G}$
is $k$-rational.
\end{enumerate}
\end{thm}

We remark that, (i) for any subgroup $G\leq S_{6}$ other than $A_{6}$,
$\mathbb{C}(x_{1},\cdots,x_{6})^{G}$ is $\mathbb{C}$-rational, and
$k(x_{1},\cdots,x_{6})^{G}$ is stably $k$-rational (for any field
$k$); (ii) for any transitive subgroup $G\leq S_{7}$ other than
$A_{7}$, $\mathbb{C}(x_{1},\cdots,x_{7})^{G}$ is $\mathbb{C}$-rational;
(iii) $k(x_{1},\cdots,x_{8})^{C_{8}}$ is $k$-rational if and only
if either $\textup{char}(k)=2$ or $[k(\zeta_{8}):k]\leq2$ (if $\textup{char}(k)\neq2$).
For details, see \cite{KW}, \cite{KWZ}, \cite[Proposition 3.9]{EM},
and \cite[Theorem 5.11]{Sa}. We also remark that the rationality
problem for the transitive subgroups of $S_{10}$ was solved and a
manuscript is in preparation.

This paper arose when we attempted to solve the rationality problem
for the subgroups of $S_{14}$. We note that this problem was investigated
in \cite{WW} for some subgroups. As noted in \cite[Theorem 3.2]{KW},
extra efforts are required in the situation when $\textup{char}(k)=p>0$
and the group order $|G|$ is divisible by $p$. If $G$ is a $p$-group,
such a ``modular'' case may be solved by the classical theorem of
Kuniyoshi and Gasch\"utz, which is recalled in the following theorem.
\begin{thm}
\label{t1.2} Suppose that $k$ is a field with $\textup{char}(k)=p>0$
and $G$ is a $p$-group.
\begin{enumerate}
\item \textup{\cite{Ku2,Ku1,Ku3}} The fixed subfield $k(G)$ is $k$-rational.
\item \textup{\cite{Ga}} For any faithful representation $\rho:G\to\textup{GL}(V)$,
where $V$ is a finite-dimensional vector space over $k$, the fixed
subfield $k(V)^{G}$ is $k$-rational.
\end{enumerate}
\end{thm}

We note that, when $G$ is cyclic and $\rho:G\to\textup{GL}(V_{\textup{reg}})$
is the regular representation, a method to find explicitly a transcendental
basis of $k(G)$ was proposed in \cite{Ha2}.

Here is a result generalizing Kuniyoshi\textendash Gasch\"utz Theorem
for the case of the regular representation (note that it is unnecessary
to assume that $G$ is a $p$-group).
\begin{thm}
\label{t1.3} \textup{\cite[Theorem 1.1]{KP}} Let $k$ be a field
with $\textup{char}(k)=p>0$, $G$ a finite group, and $\tilde{G}$
a group extension given by $1\to\mathbb{Z}/p\mathbb{Z}\to\tilde{G}\to G\to1$.
Then $k(\tilde{G})$ is rational over $k(G)$.
\end{thm}

The referee told us that Section 4 of Saltman's paper \cite{Sa} contains
results similar to the above Theorem \ref{t1.3}.

However, when we consider rationality problems of subgroups in $S_{14}$
with the field $k$ satisfying that $\textup{char}(k)=7$, new situations
may arise for which Theorem \ref{t1.2} and \ref{t1.3} will not work
any more. For example, here is a typical question.
\begin{question}
\label{q1.4} Let $k$ be a field with $\textup{char}(k)=7$, and
$k(t,x_{1},\cdots,x_{6})$ the rational function field of $7$ variables
over $k$ (i.e., $\textup{trdeg}_{k}k(t,x_{1},\cdots,x_{6})=7$).
Suppose that $G:=\left\langle \sigma,\tau\right\rangle $ is a group
acting on $k(t,x_{1},\cdots,x_{6})$ by the monomial $k$-automorphisms
defined by 
\begin{align*}
 & \sigma:t\mapsto t,x_{1}\mapsto x_{2}\mapsto\cdots\mapsto x_{6}\mapsto t^{2}/{\textstyle \prod_{i=1}^{6}}x_{i},\\
 & \tau:t\mapsto-t,x_{i}\mapsto x_{3i},
\end{align*}
where the indices of $x_{i}$ are taken modulo $7$. Is the fixed
subfield $k(t,x_{1},\cdots,x_{6})^{G}$ $k$-rational? Is it stably
$k$-rational?
\end{question}

The reader may find the definition of monomial $k$-automorphisms
in \cite[page 805]{HK1}.

This example prompts us to reexamine Kuniyoshi\textendash Gasch\"utz
Theorem and to improve the old techniques for this new situation.
As an illustration, the following theorem, which is a generalization
of \cite[Lemma 2.7]{WZ}, will solve Question \ref{q1.4} affirmatively
for the stable rationality.
\begin{thm}
\label{t1.5} Let $k$ be a field, $n$ a positive integer, $d$ an
integer relatively prime to $n$, $k(t,s,x_{0},x_{1},\cdots,x_{n-1})$
a field satisfying that $\textup{trdeg}_{k}k(t,s,x_{0},x_{1},\cdots,x_{n-1})=n+1$
and $t^{d}=\prod_{i=0}^{n-1}x_{i}$. Let $G$ be a subgroup of $S_{n}$
and $\chi:G\to k^{\times}$ a linear character with $\chi^{d}=1$
(the possibility that $\chi$ is the trivial character is allowed).
Suppose that $G$ acts on $k(t,s,x_{0},x_{1},\cdots,x_{n-1})$ by
$k$-automorphisms defined by $\sigma\cdot t:=\chi(\sigma)t$, $\sigma\cdot s:=s$,
$\sigma\cdot x_{i}:=x_{\sigma\cdot i}$ for any $\sigma\in G$ and
$0\leq i\leq n-1$. Then there exists some element $u\in k(t,s,x_{0},x_{1},\cdots,x_{n-1})$
such that $\sigma\cdot u=u$ for any $\sigma\in G$, and $k(t,s,x_{0},x_{1},\cdots,x_{n-1})^{G}=k(x_{0}/s,x_{1}/s,\cdots,x_{n-1}/s)^{G}(u)$.
\end{thm}

Note that the field $k(t,s,x_{0},x_{1},\cdots,x_{n-1})$ with $t^{d}=\prod_{i=0}^{n-1}x_{i}$
is nothing but the rational function field $k(t,s,x_{1},\cdots,x_{n-1})$
in the variables $t,s,x_{1},\cdots,x_{n-1}$. With Theorem \ref{t1.5}
in hand, we may solve Question \ref{q1.4} by introducing a new element
$x_{0}$ such that $t^{2}=\prod_{i=0}^{6}x_{i}$. Applying Theorem
\ref{t1.5}, we find that the rationality problem of $k(t,x_{1},\cdots,x_{6})^{G}(s)$
is reduced to that of $k(x_{0}/s,x_{1}/s,\cdots,x_{6}/s)^{G}(u)$.
Then we may apply Theorem \ref{t1.1}.

In Theorem \ref{t1.5}, we may as well consider the case $k(t,s,x_{0},x_{1},\cdots,x_{p-1})$
is a field satisfying that, $p$ is a prime number, $\textup{trdeg}_{k}k(t,s,x_{0},x_{1},\cdots,x_{p-1})=p+1$,
and $t^{d'}=\prod_{i=0}^{p-1}x_{i}$ with $d'$ divisible by $p$.
See Question \ref{q3.1}.

Before stating the next result, we digress to define some $k$-automorphisms.
\begin{defn}
\label{d1.6} Let $k$ be a field, $p$ an odd prime number, $a$
an integer satisfying that $(\mathbb{Z}/p\mathbb{Z})^{\times}=\langle\bar{a}\rangle$,
$S_{p}$ the symmetric group of degree $p$, and $k(x_{i},y_{i}:0\leq i\leq p-1)$
the rational function field of $2p$ variables over $k$. We define
six elements $\sigma_{1},\sigma_{2},\lambda_{1},\lambda_{2},\rho_{1},\rho_{2}\in S_{p}$
such that they act on $k(x_{i},y_{i}:0\leq i\leq p-1)$ by
\begin{align*}
 & \sigma_{1}:x_{i}\mapsto x_{i+1},y_{i}\mapsto y_{i}, &  & \sigma_{2}:x_{i}\mapsto x_{i},y_{i}\mapsto y_{i+1}, &  & \lambda_{1}:x_{i}\mapsto y_{-i},y_{i}\mapsto x_{i},\\
 & \lambda_{2}:x_{i}\leftrightarrow y_{i}, &  & \rho_{1}:x_{i}\mapsto x_{ai},y_{i}\mapsto y_{i}, &  & \rho_{2}:x_{i}\mapsto x_{i},y_{i}\mapsto y_{ai},
\end{align*}
where $0\leq i\leq p-1$, and the indices of $x_{i}$ and $y_{i}$
are taken modulo $p$.
\end{defn}

Note that $\rho_{1}$ and $\rho_{2}$ depend on $a$. In the proof
of Section \ref{sec5.3}, we will use the same $\rho_{1}$ or $\rho_{2}$,
but the values of $a$ will be different.

Let $\sigma_{1},\sigma_{2},\lambda_{1},\lambda_{2},\rho_{1},\rho_{2}$
be the $k$-automorphisms introduced in Definition \ref{d1.6}. Here
is another theorem to be used in Section \ref{sec5}.
\begin{thm}
\label{t1.7} Let $p$ be an odd prime number, $k$ a field with $\textup{char}(k)=p>0$,
and $k(x_{i},y_{i}:0\leq i\leq p-1)$ the rational function field
of $2p$ variables over $k$. Then $k(x_{i},y_{i}:0\leq i\leq p-1)^{G}$
is $k$-rational when $G$ is $\langle\sigma_{1},\sigma_{2},\lambda_{1},\rho_{1}\rho_{2}\rangle$,
$\langle\sigma_{1},\sigma_{2},\lambda_{1},\rho_{1}^{-1}\rho_{2}\rangle$,
$\langle\sigma_{1},\sigma_{2},\lambda_{2},\rho_{1}\rho_{2}\rangle$,
or $\langle\sigma_{1},\sigma_{2},\lambda_{2},\rho_{1}^{-1}\rho_{2}\rangle$.
\end{thm}

As an application of Theorem \ref{t1.5} and \ref{t1.7}, we will
prove the following theorem.
\begin{thm}
\label{t1.8} Let $k$ be a field with $\textup{char}(k)=7$ and $G$
a solvable transitive subgroup of $S_{14}$ acting naturally on $k(x_{1},\cdots,x_{14})$,
the rational function field of $14$ variables over $k$. Then $k(x_{1},\cdots,x_{14})^{G}$
is k-rational.
\end{thm}

Another version of Theorem \ref{t1.8} may be found in Theorem \ref{t4.1},
which employs a description of the transitive subgroups of $S_{14}$.
In Theorem \ref{t1.8}, if $G$ is a non-solvable transitive subgroup
of $S_{14}$, the rationality of $k(x_{1},\cdots,x_{14})^{G}$ is
known only for a few cases (see Theorem \ref{t4.1} for details).

We remark that Theorem \ref{t1.5} and \ref{t1.7} may be applied
also to prove the rationality problem of $k(x_{1},\cdots,x_{10})^{G}$
when $\textup{char}(k)=5$ and $G$ is a solvable transitive subgroup
of $S_{10}$. It may be true as well for many cases of $k(x_{1},\cdots,x_{2p})^{G}$
when $\textup{char}(k)=p$ is an odd prime number.

The ``non-modular'' situation of Theorem \ref{t1.8} (i.e., for
the case $\textup{char}(k)\neq7$) is an ongoing research program.
Hopefully it will be finished very soon.

This article is organized as follows. Section \ref{sec2} contains
some known results which will be applied in Section \ref{sec3} and
\ref{sec5}. The proofs of Theorem \ref{t1.5} and \ref{t1.7} will
be given in Section \ref{sec3}. A list of transitive subgroups of
$S_{14}$ (up to conjugation within $S_{14}$) is provided in Section
\ref{sec4}. Theorem \ref{t4.1} is a sharp form of Theorem \ref{t1.8}
and the proof of Theorem \ref{t4.1} may be found in Section \ref{sec5}.

\textit{Standing terminology.} Throughout this paper, $G$ is a finite
group and $k$ is a field. Recall that $k(G)$ is defined at the beginning
of this section. We will denote by $S_{n}$, $A_{n}$, $C_{n}$, and
$D_{n}$ the symmetric group of degree $n$, the alternating group
of degree $n$, the cyclic group of order $n$, and the dihedral group
of order $2n$ respectively. When we say that $k(x_{1},\cdots,x_{n})$
is a rational function field over a field $k$, we mean that $k(x_{1},\cdots,x_{n})$
is purely transcendental over $k$, and $\{x_{1},\cdots,x_{n}\}$
is a transcendental basis, equivalently, $\textup{trdeg}_{k}k(x_{1},\cdots,x_{n})=n$.
If $\sigma:k(x_{1},\cdots,x_{n})\to k(x_{1},\cdots,x_{n})$ is an
automorphism and $u\in k(x_{1},\cdots,x_{n})$, then $\sigma\cdot u$
denotes $\sigma(u)$, the image of $u$ under $\sigma$. By a linear
character $\chi:G\to k^{\times}$ we mean a group homomorphism from
$G$ to $k^{\times}$. Whenever we write $\chi:G\to k^{\times}$,
it is assumed that, for any $\sigma\in G$, $\chi(\sigma)\in k$ automatically.

Although we work on the case $\textup{char}(k)=p>0$ sometimes, we
also work on general cases. We will always state the assumptions of
the field $k$ explicitly. The readers should be aware not to confuse
the groups $G(21)$ and $G_{21}$: $G(21)$ is the $21$st transitive
subgroup of $S_{14}$ in Section \ref{sec4}, while $G_{21}$ is the
group $G_{pd}$ with $(p,d)=(7,3)$ and is defined in Definition \ref{d2.3}.

\textbf{Acknowledgments.} We thank the referee whose sharp and insightful
comments are very helpful to the presentation of this article.

\section{\label{sec2}Preliminaries}

In this section, we recall several known results. These results will
be used in the proofs of Section \ref{sec3} and \ref{sec5}.
\begin{thm}
\label{t2.1} \textup{\cite[Theorem 1]{HK2}} Let $G$ be a finite
group acting on $L(x_{1},\cdots,x_{n})$, the rational function field
of $n$ variables over a field $L$. Suppose that
\begin{itemize}
\item for any $\sigma\in G$, $\sigma(L)\subseteq L$; 
\item the restriction of the action of $G$ to $L$ is faithful; 
\item for any $\sigma\in G$, 
\[
\begin{pmatrix}\sigma(x_{1})\\
\vdots\\
\sigma(x_{n})
\end{pmatrix}=A(\sigma)\begin{pmatrix}x_{1}\\
\vdots\\
x_{n}
\end{pmatrix}+B(\sigma),
\]
where $A(\sigma)\in\textup{GL}_{n}(L)$ and $B(\sigma)$ is an $n\times1$
matrix over $L$. 
\end{itemize}
Then there exist $z_{1},\cdots,z_{n}\in L(x_{1},\cdots,x_{n})$ such
that $L(x_{1},\cdots,x_{n})=L(z_{1},\cdots,z_{n})$ with $\sigma(z_{i})=z_{i}$
for any $\sigma\in G$ and $1\leq i\leq n$.
\end{thm}

\begin{thm}
\label{t2.2} \textup{\cite[Theorem 3.1]{AHK}} Let $G$ be a finite
group acting on $L(x)$, the rational function field of one variable
over a field $L$. Suppose that for any $\sigma\in G$, $\sigma(L)\subseteq L$
and $\sigma(x)=a_{\sigma}x+b_{\sigma}$, where $a_{\sigma}\in L^{\times}$
and $b_{\sigma}\in L$. Then $L(x)^{G}=L^{G}(f)$ for some polynomial
$f\in L[x]^{G}$.
\end{thm}

\begin{defn}
\label{d2.3} \cite[Definition 3.1]{KW} Let $k$ be a field, $p$
an odd prime number, $a$ an integer satisfying that $(\mathbb{Z}/p\mathbb{Z})^{\times}=\langle\bar{a}\rangle$,
and $k(x_{0},x_{1},\cdots,x_{p-1})$ the rational function field of
$p$ variables over $k$. Define the $k$-automorphisms $\sigma,\tau$
on $k(x_{0},x_{1},\cdots,x_{p-1})$ by $\sigma:x_{i}\mapsto x_{i+1}$
and $\tau:x_{i}\mapsto x_{ai}$, where the indices of $x_{i}$ are
taken modulo $p$. Let $d$ be a positive divisor of $p-1$ and write
$p-1=de$. Define a group $G_{pd}:=\langle\sigma,\tau^{e}\rangle$.
As abstract groups, $G_{p}\simeq C_{p}$ the cyclic group (when $d=1$),
$G_{2p}\simeq D_{p}$ the dihedral group (when $d=2$), and $G_{p(p-1)}$
is isomorphic to the maximal solvable transitive subgroup of the symmetric
group $S_{p}$. In general, $G_{pd}$ is a semidirect product.
\end{defn}

\begin{thm}
\label{t2.4} \textup{\cite[Theorem 3.2]{KW}} Let $k(x_{0},x_{1},\cdots,x_{p-1})$
and $G_{pd}$ be the same as in Definition \ref{d2.3}. If $\textup{char}(k)=p>0$,
then the fixed subfield $k(x_{0},x_{1},\cdots,x_{p-1})^{G_{pd}}$
is $k$-rational.
\end{thm}

\begin{defn}
\label{d2.5} \cite[Definition 3.1 and 3.2]{KWZ} Let $G$ and $H$
be finite groups such that $G$ acts on a finite set $X$ from the
left. Let $A$ be the set of all functions from $X$ to $H$, then
$G$ acts naturally on $A$ by $(g\cdot a)(x):=a(g^{-1}\cdot x)$,
where $g\in G$, $a\in A$, and $x\in X$. Note that $A$ may be identified
with the direct product of $|X|$ copies of $H$, thus any $a\in A$
may be written as $a=(a_{x}:x\in X)$. Under this identification,
the group $G$ acts on $A$ by $(g\cdot a)_{x}:=a_{g^{-1}\cdot x}$
where $g\in G$, $a\in A$, and $(g\cdot a)_{x}$ is the $x$-component
of $g\cdot a$.

The \textit{wreath product} $H\wr_{X}G$ is defined as the semidirect
product $A\rtimes G$, where $A$ is the normal subgroup with an action
of $G$ from the left.

Furthermore, when $G$ and $H$ are groups acting on the sets $X$
and $Y$ from the left, the wreath product $H\wr_{X}G:=A\rtimes G$
acts on $Y\times X$ by defining $(a,g)\cdot(y,x):=(a_{g\cdot x}\cdot y,g\cdot x)$
for any $g\in G$, $a\in A$, $x\in X$, and $y\in Y$.
\end{defn}

Adopting the notation of Definition \ref{d2.5}, we write $X_{m}:=\{1,\cdots,m\}$
(the set consisting of $m$ elements), and $Y_{n}:=\{1,\cdots,n\}$
(the set consisting of $n$ elements). If $G\leq S_{m}$ and $H\leq S_{n}$,
then we may regard $H\wr_{X_{m}}G$ as a subgroup of $S_{mn}$ because
$Y_{n}\times X_{m}$ is a set consisting of $mn$ elements. With this
understanding, we have the following theorem.
\begin{thm}
\label{t2.6} \textup{\cite[Theorem 3.5]{KWZ}} Let $k$ be a field,
$G\leq S_{m}$ and $H\leq S_{n}$ act on the rational function fields
$k(x_{1},\cdots,x_{m})$ and $k(y_{1},\cdots,y_{n})$ respectively.
Assume that both $k(x_{1},\cdots,x_{m})^{G}$ and $k(y_{1},\cdots,y_{n})^{H}$
are $k$-rational. Then the wreath product $H\wr_{X_{m}}G$ acts on
the rational function field $k(z_{i,j}:1\leq i\leq m,1\leq j\leq n)$,
and the fixed subfield $k(z_{i,j}:1\leq i\leq m,1\leq j\leq n)^{H\wr_{X_{m}}G}$
is $k$-rational.
\end{thm}

In a similar way, if $G$ and $H$ act on the sets $X$ and $Y$ respectively
from the left, then the direct product $G\times H$ acts on $X\times Y$
by $(g,h)\cdot(x,y):=(g\cdot x,h\cdot y)$.
\begin{thm}
\label{t2.7} \textup{\cite[Theorem 3.6]{KWZ}} Let $k$, $G$, and
$H$ be the same as in Theorem \ref{t2.6}. Then the direct product
$G\times H$ acts on the rational function field $k(z_{i,j}:1\leq i\leq m,1\leq j\leq n)$
such that the fixed subfield $k(z_{i,j}:1\leq i\leq m,1\leq j\leq n)^{G\times H}$
is $k$-rational.
\end{thm}

\begin{lem}
\label{l2.8} \textup{\cite[page 245]{Ha1}} Let $k$ be a field,
$k(x_{1},\cdots,x_{n-1})$ the rational function field of $n-1$ variables
over $k$. Suppose that $\sigma$ is a $k$-automorphism on $k(x_{1},\cdots,x_{n-1})$
defined by $\sigma:x_{1}\mapsto x_{2}\mapsto\cdots\mapsto x_{n-1}\mapsto1/\prod_{i=1}^{n-1}x_{i}$.
Then there exist $y_{1},\cdots,y_{n-1}\in k(x_{1},\cdots,x_{n-1})$
such that $k(x_{1},\cdots,x_{n-1})=k(y_{1},\cdots,y_{n-1})$ and $\sigma:y_{1}\mapsto y_{2}\mapsto\cdots\mapsto y_{n-1}\mapsto1-\sum_{i=1}^{n-1}y_{i}$.
\end{lem}

\begin{proof}
For the convenience of the reader, we include the proof of \cite{Ha1}.
Define $w:=1+x_{1}+x_{1}x_{2}+\cdots+x_{1}x_{2}\cdots x_{n-1}$, $y_{1}:=1/w$,
and $y_{i}:=x_{1}x_{2}\cdots x_{i-1}/w$ for $2\leq i\leq n$. It
follows that $\sum_{i=1}^{n}y_{i}=1$ and $\sigma:y_{1}\mapsto y_{2}\mapsto\cdots\mapsto y_{n}\mapsto y_{1}$.
\end{proof}
\begin{lem}
\label{l2.9} Let $k$ be a field, $p$ an odd prime number, $k(y_{1},\cdots,y_{p-1})$
the rational function field of $p-1$ variables over $k$. Suppose
that $\sigma$ is a $k$-automorphism on $k(y_{1},\cdots,y_{p-1})$
defined by $\sigma:y_{1}\mapsto y_{2}\mapsto\cdots\mapsto y_{p-1}\mapsto1-\sum_{i=1}^{p-1}y_{i}$.
\begin{enumerate}
\item If $\textup{char}(k)=p$, then the fixed subfield $k(y_{1},\cdots,y_{p-1})^{\left\langle \sigma\right\rangle }$
is $k$-rational.
\item If $\textup{char}(k)\neq p$, then the field $k(y_{1},\cdots,y_{p-1})^{\left\langle \sigma\right\rangle }(z)$
is $k$-isomorphic to $k(C_{p})$, where $z$ is an element transcendental
over $k(y_{1},\cdots,y_{p-1})^{\left\langle \sigma\right\rangle }$.
\end{enumerate}
\end{lem}

\begin{proof}
(i) Define $u:=\sum_{i=1}^{p-1}(p-i)y_{i}$. Then $\sigma:u\mapsto u+1$.
Apply Theorem \ref{t2.1} to $k(y_{1},\cdots,y_{p-1})=L(y_{2},\cdots,y_{p-1})$
with $L:=k(u)$. Since $L^{\left\langle \sigma\right\rangle }$ is
$k$-rational by L\"uroth's theorem, done.

(ii) Since $1/p$ exists and belongs to $k$, we may define $z_{i}:=y_{i}-1/p$
for $1\leq i\leq p-1$. It follows that $\sigma:z_{1}\mapsto z_{2}\mapsto\cdots\mapsto z_{p-1}\mapsto-\sum_{i=1}^{p-1}z_{i}$.
On the other hand, write $k(C_{p}):=k(x_{0},x_{1},\cdots,x_{p-1})^{\left\langle \sigma\right\rangle }$,
where $\sigma:x_{i}\mapsto x_{i+1}$ (the indices of $x_{i}$ are
taken modulo $p$). Define $v:=\sum_{i=0}^{p-1}x_{i}$ and $v_{i}:=x_{i}-v/p$
for $0\leq i\leq p-1$. Note that $\sum_{i=0}^{p-1}v_{i}=0$. Define
$K:=k(v_{0},v_{1},\cdots,v_{p-1})$. Then $k(x_{0},x_{1},\cdots,x_{p-1})=K(v)$
and apply Theorem \ref{t2.2}. Hence the result.
\end{proof}

\section{\label{sec3}Proofs of Theorem \ref{t1.5} and \ref{t1.7}}

We will prove Theorem \ref{t1.5} and \ref{t1.7} in this section.
\begin{proof}
[Proof of Theorem \ref{t1.5}] The action of $G$ on $k(t,s,x_{0},x_{1},\cdots,x_{n-1})$
is well-defined because $\sigma\cdot t^{d}=\chi^{d}(\sigma)t^{d}=t^{d}$
for any $\sigma\in G$. Let $H:=\ker(\chi)$. Then
\begin{eqnarray*}
 &  & k(t,s,x_{0},x_{1},\cdots,x_{n-1})^{G}\\
 & = & k(t,s,x_{0}/s,x_{1}/s,\cdots,x_{n-1}/s)^{G}\\
 & = & k(t^{d}/s^{n},t^{a}s^{b},x_{0}/s,x_{1}/s,\cdots,x_{n-1}/s)^{G}\text{ for some }a,b\in\mathbb{Z}\text{ such that }an+bd=1\\
 & = & k(t^{a}s^{b},x_{0}/s,x_{1}/s,\cdots,x_{n-1}/s)^{G}\text{ since }t^{d}/s^{n}={\textstyle \prod}_{i=0}^{n-1}(x_{i}/s)\\
 & = & \left(k(t^{a}s^{b},x_{0}/s,x_{1}/s,\cdots,x_{n-1}/s)^{H}\right)^{G/H}\\
 & = & \left(k(x_{0}/s,x_{1}/s,\cdots,x_{n-1}/s)^{H}(t^{a}s^{b})\right)^{G/H}\text{ since }H:=\ker(\chi)\text{ fixes }t^{a}s^{b}\\
 & = & \left(k(x_{0}/s,x_{1}/s,\cdots,x_{n-1}/s)^{H}\right)^{G/H}(u)\text{ by applying Theorem \ref{t2.2}}\\
 & = & k(x_{0}/s,x_{1}/s,\cdots,x_{n-1}/s)^{G}(u).
\end{eqnarray*}
\end{proof}
\begin{proof}
[Proof of Theorem \ref{t1.7}] Note that $\lambda_{1}^{4}=\lambda_{2}^{2}=1$.
Moreover, it is straightforward to check that
\begin{align*}
 & \lambda_{1}\sigma_{1}\lambda_{1}^{-1}=\sigma_{2}^{-1}, &  & \lambda_{1}\sigma_{2}\lambda_{1}^{-1}=\sigma_{1}, &  & \lambda_{2}\sigma_{1}\lambda_{2}^{-1}=\sigma_{2}, &  & \lambda_{2}\sigma_{2}\lambda_{2}^{-1}=\sigma_{1},\\
 & \rho_{1}\sigma_{1}\rho_{1}^{-1}=\sigma_{1}^{a}, &  & \rho_{1}\sigma_{2}\rho_{1}^{-1}=\sigma_{2}, &  & \rho_{2}\sigma_{1}\rho_{2}^{-1}=\sigma_{1}, &  & \rho_{2}\sigma_{2}\rho_{2}^{-1}=\sigma_{2}^{a}.
\end{align*}
Define $N:=\langle\sigma_{1},\sigma_{2}\rangle\leq G$, where $G$
is one of the four groups we consider. Clearly $N$ is a normal subgroup
of $G$. Define
\[
u_{i}:=\sum_{j=0}^{p-1}j^{i}x_{j}\qquad\text{ and }\qquad v_{i}:=\sum_{j=0}^{p-1}j^{i}y_{j},
\]
where $0\leq i\leq p-1$ (by convention we write $0^{0}=1$). We find
that
\[
k(x_{i},y_{i}:0\leq i\leq p-1)=k(u_{i},v_{i}:0\leq i\leq p-1),
\]
and
\begin{align*}
 & \sigma_{1}\cdot u_{i}=\sum_{j=0}^{p-1}j^{i}x_{j+1}=\sum_{j=0}^{p-1}(j-1)^{i}x_{j}, &  & \sigma_{1}\cdot v_{i}=v_{i},\\
 & \sigma_{2}\cdot u_{i}=u_{i}, &  & \sigma_{2}\cdot v_{i}=\sum_{j=0}^{p-1}j^{i}y_{j+1}=\sum_{j=0}^{p-1}(j-1)^{i}y_{j}.
\end{align*}
Consequently,
\begin{align*}
 & \sigma_{1}:u_{0}\mapsto u_{0},u_{1}\mapsto u_{1}-u_{0},u_{2}\mapsto u_{2}-2u_{1}+u_{0},\cdots,\\
 & \sigma_{2}:v_{0}\mapsto v_{0},v_{1}\mapsto v_{1}-v_{0},v_{2}\mapsto v_{2}-2v_{1}+v_{0},\cdots.
\end{align*}
Also note that
\[
\begin{aligned} & \lambda_{1}:u_{i}\mapsto(-1)^{i}v_{i},v_{i}\mapsto u_{i}, &  & \lambda_{2}:u_{i}\leftrightarrow v_{i}, &  & \rho_{1}:u_{i}\mapsto a^{-i}u_{i}, &  & \rho_{2}:v_{i}\mapsto a^{-i}v_{i}.\end{aligned}
\]
We will show that $k(x_{i},y_{i}:0\leq i\leq p-1)^{G}$ is $k$-rational.
By Theorem \ref{t2.1}, it suffices to show that $k(u_{0},v_{0},u_{1},v_{1})^{G}$
is $k$-rational. Recall that $N:=\langle\sigma_{1},\sigma_{2}\rangle\leq G$,
and
\[
\begin{aligned} & \sigma_{1}:\frac{u_{1}}{u_{0}}\mapsto\frac{u_{1}}{u_{0}}-1,\frac{v_{1}}{v_{0}}\mapsto\frac{v_{1}}{v_{0}}, &  & \sigma_{2}:\frac{u_{1}}{u_{0}}\mapsto\frac{u_{1}}{u_{0}},\frac{v_{1}}{v_{0}}\mapsto\frac{v_{1}}{v_{0}}-1.\end{aligned}
\]
We have
\[
k(u_{0},v_{0},u_{1},v_{1})^{N}=k\left(u_{0},v_{0},\frac{u_{1}}{u_{0}},\frac{v_{1}}{v_{0}}\right)^{N}=k(u_{0},v_{0},u_{1}',v_{1}'),
\]
where $u_{1}'$ and $v_{1}'$ are the Artin\textendash Schreier elements
defined by
\[
u_{1}':=\left(\frac{u_{1}}{u_{0}}\right)^{p}-\frac{u_{1}}{u_{0}}\qquad\text{ and }\qquad v_{1}':=\left(\frac{v_{1}}{v_{0}}\right)^{p}-\frac{v_{1}}{v_{0}}.
\]
Note that
\[
\left[k\left(u_{0},v_{0},\frac{u_{1}}{u_{0}},\frac{v_{1}}{v_{0}}\right):k(u_{0},v_{0},u_{1}',v_{1}')\right]=p^{2}=|N|,
\]
because
\[
u_{1}'=\prod_{i=0}^{p-1}\left(\frac{u_{1}}{u_{0}}+i\right)\qquad\text{ and }\qquad v_{1}'=\prod_{i=0}^{p-1}\left(\frac{v_{1}}{v_{0}}+i\right).
\]
Now we have
\[
\begin{aligned} & \lambda_{1}:u_{0}\leftrightarrow v_{0},u_{1}'\mapsto-v_{1}',v_{1}'\mapsto u_{1}', &  & \lambda_{2}:u_{0}\leftrightarrow v_{0},u_{1}'\leftrightarrow v_{1}', &  & \rho_{1}:u_{1}'\mapsto a^{-1}u_{1}', &  & \rho_{2}:v_{1}'\mapsto a^{-1}v_{1}'.\end{aligned}
\]
To show that $k(u_{0},v_{0},u_{1}',v_{1}')^{G/N}$ is $k$-rational,
we apply Theorem \ref{t2.1} again. It remains to show that $k(u_{1}',v_{1}')^{G/N}$
is $k$-rational. By applying Theorem \ref{t2.2} to $k(u_{1}',v_{1}')=k(u_{1}'/v_{1}')(v_{1}')$,
the problem is reduced to the rationality of $k(u_{1}'/v_{1}')^{G/N}$,
which is the case because of L\"uroth's theorem.

To find the explicit generator we illustrate the case $G=\left\langle \sigma_{1},\sigma_{2},\lambda_{1},\rho_{1}\rho_{2}\right\rangle $.
Since $N=\left\langle \sigma_{1},\sigma_{2}\right\rangle $, it follows
that $G/N=\left\langle \lambda_{1},\rho_{1}\rho_{2}\right\rangle $.
Note that the action of $G/N$ on $k(u_{1}'/v_{1}')$ is not faithful.
In fact, $\rho_{1}\rho_{2}$ acts trivially on it and $\lambda_{1}:u_{1}'/v_{1}'\mapsto-v_{1}'/u_{1}'$.
We conclude that the fixed subfield $k(u_{1}'/v_{1}')^{G/N}=k(u_{1}'/v_{1}'-v_{1}'/u_{1}')$.
\end{proof}
Consider a variant of Theorem \ref{t1.5}. Let $k$ be a field, $p$
an odd prime number, $k(t,s,x_{0},x_{1},\cdots,x_{p-1})$ a field
satisfying that $\textup{trdeg}_{k}k(t,s,x_{0},x_{1},\cdots,x_{p-1})=p+1$
and $t^{d'}=\prod_{i=0}^{p-1}x_{i}$, where $d'$ is an integer divisible
by $p$. Write $d'=pd_{0}$ and define $X_{i}:=x_{i}/t^{d_{0}}$,
then $\prod_{i=0}^{p-1}X_{i}=1$. Consider the action of the group
$G_{pd}$ in Definition \ref{d2.3}. The rationality problem of $k(t,s,X_{0},X_{1},\cdots,X_{p-1})^{G_{pd}}$
is reduced, up to stable rationality, to the following question.
\begin{question}
\label{q3.1} Let $k$ be a field, $p$ an odd prime number, $k(x_{1},\cdots,x_{p-1})$
the rational function field of $p-1$ variables over $k$. Let $a$
be an integer satisfying that $(\mathbb{Z}/p\mathbb{Z})^{\times}=\langle\bar{a}\rangle$.
Define $k$-automorphisms $\sigma,\tau$ on $k(x_{1},\cdots,x_{p-1})$
by $\sigma:x_{1}\mapsto x_{2}\mapsto\cdots\mapsto x_{p-1}\mapsto1/\prod_{i=1}^{p-1}x_{i}$
and $\tau:x_{i}\mapsto x_{ai}$, where the indices of $x_{i}$ are
taken modulo $p$. For any positive divisor $d$ of $p-1$, write
$p-1=de$ and define the group $G_{pd}:=\langle\sigma,\tau^{e}\rangle$.
We ask under what situations the fixed subfield $k(x_{1},\cdots,x_{p-1})^{G_{pd}}$
is $k$-rational (resp., stably $k$-rational).
\end{question}

Alternatively, we may formulate the above question in terms of $G$-lattices.
The reader may consult Section 2 of \cite{HKY} for unexplained terminology.

Let $p$ be an odd prime number and $a$ an integer satisfying that
$(\mathbb{Z}/p\mathbb{Z})^{\times}=\left\langle \bar{a}\right\rangle $.
Define the group $G$ by $G:=\left\langle \sigma,\tau\right\rangle $,
where $\sigma^{p}=\tau^{p-1}=1$ and $\tau\sigma\tau^{-1}=\sigma^{a}$.
For any positive divisor $d$ of $p-1$, write $p-1=de$ and define
the group $G_{pd}:=\langle\sigma,\tau^{e}\rangle$.

Now define the lattices $N:=\mathbb{Z}[G/\left\langle \tau\right\rangle ]$
and $M:=N/(\mathbb{Z}\sum_{i=0}^{p-1}\sigma^{i})$. In other words,
$0\to\mathbb{Z}\to N\to M\to0$ is a short exact sequence of $G$-lattices.
Note that $N$ may be described explicitly. Write $N':=\oplus_{i=0}^{p-1}\mathbb{Z}e_{i}$
with the $G$-action defined by $\sigma\cdot e_{i}=e_{i+1}$ and $\tau\cdot e_{i}=e_{ai}$.
Then $N$ and $N'$ are isomorphic $G$-lattices.

If $k$ is a field, consider $k(N)^{G_{pd}}$ and $k(M)^{G_{pd}}$
where the group acts trivially on the field $k$. Questions: Is $k(N)^{G_{pd}}$
rational? Is $k(M)^{G_{pd}}$ rational? Is $[M]^{fl}$ a permutation
lattice or an invertible lattice?

We remark that, if $d=1$ (i.e., $G_{p}\cong C_{p}$, the cyclic group
of order $p$), then $k(M)^{G_{p}}$ is stably $k$-rational if and
only if so is $k(G_{p})$, because we may apply Lemma \ref{l2.8}
and \ref{l2.9}.

If $d=2$ (i.e., $G_{2p}\cong D_{p}$, the dihedral group of order
$2p$), then $k(M)^{G_{2p}}$ is stably $k$-rational if and only
if so is $k(G_{2p})$ by \cite[Theorem 5.6]{HKY}.

On the other hand, we don't know the answer whether $k(M)^{G_{pd}}$
is $k$-rational (resp., stably $k$-rational) if $d$ is a divisor
of $p-1$ other than $1$ or $2$. The situation of $k(N)^{G_{pd}}$
is unclear even when $k=\mathbb{Q}$, which was investigated by Breuer
during 1920s \cite{Br}.

\section{\label{sec4}Transitive Subgroups of \texorpdfstring{$\bm{S_{14}}$}{$S_{14}$}}

The transitive subgroups of $S_{14}$, up to conjugation within $S_{14}$,
was classified by \cite{Mi}. It was reconfirmed in \cite{CHM}. The
message was finally integrated into the libraries of \cite{GAP}.
In this paper, we will name these groups as those given in the GAP
library.

The symmetric group $S_{14}$ contains $63$ transitive subgroups.
They are labeled as $G(i)$, where $1\leq i\leq63$. Among the $63$
subgroups, exactly $36$ of them are solvable.

These subgroups are classified into $13$ classes. Groups belonging
to Class $1$\textendash $6$ are solvable, while groups belonging
to Class $7$\textendash $13$ are non-solvable. Note that the classification
into $13$ classes is devised by us, not by the GAP library.

To describe these subgroups, we first define the following elements
in $S_{14}$:
\begin{eqnarray*}
\sigma_{1} & = & (1,3,5,7,9,11,13),\\
\sigma_{2} & = & (2,4,6,8,10,12,14),\\
\tau_{1} & = & (1,2)(3,8,5,14,9,12)(4,7,6,13,10,11),\\
\tau_{2} & = & (3,7,5,13,9,11)(4,8,6,14,10,12),\\
\lambda_{1} & = & (1,2)(3,14,13,4)(5,12,11,6)(7,10,9,8),\\
\lambda_{2} & = & (1,2)(3,4)(5,6)(7,8)(9,10)(11,12)(13,14),\\
\lambda_{3} & = & (3,13)(5,11)(7,9),\\
\lambda_{4} & = & (4,14)(6,12)(8,10),\\
\lambda_{5} & = & (3,5,9)(7,13,11),\\
\lambda_{6} & = & (4,6,10)(8,14,12),\\
\mu_{i} & = & (2i+1,2i+2),\text{ for }0\leq i\leq6,\\
\nu_{1} & = & (1,4)(2,3)(5,6)(7,8)(9,10)(11,12)(13,14),\\
\nu_{2} & = & (1,3)(2,4),\\
\nu_{3} & = & (1,6)(2,5)(3,11)(4,12)(7,8)(9,10).
\end{eqnarray*}

Here is a list of these $63$ subgroups $G(i)$ of $S_{14}$.

The notation $G$ refers to some group $G(i)$ when we don't intend
to specify its numeral. Remember that the groups $G_{21}$ and $G_{42}$
are subgroups of $S_{7}$ in Definition \ref{d2.3}, don't confuse
them with $G(21)$ and $G(42)$. We also remark that in this list,
some groups may be isomorphic as abstract groups (we use the notation
$\simeq$), but they are not conjugate within $S_{14}$. For example,
consider the group $G(17)$ in Class $12$ and the group $G(19)$
in Class $11$. We write $G(17)\simeq C_{2}\times\textup{PSL}_{2}(\mathbb{F}_{7})$
to mean that $G(17)$ is isomorphic to a direct product of $C_{2}$
and $\textup{PSL}_{2}(\mathbb{F}_{7})$ as abstract groups. On the
other hand, we write $G(19)=\textup{PSL}_{2}(\mathbb{F}_{7})\times C_{2}$
to mean that $G(19)$ is defined as a subgroup of $S_{14}$, where
the first factor is a subgroup of $S_{7}$, the second factor is $S_{2}$,
and the identification of the direct product as a subgroup of $S_{14}$
is understood as the way given in Theorem \ref{t2.7}. We adopt similar
convention for groups in Classes $1$, $8$, and $11$ (see the explanation
in Definition \ref{d2.5}, Theorem \ref{t2.6}, and \ref{t2.7}).

\textbf{Class $\bm{1}$.} The direct product and the wreath product
of smaller groups:
\[
\begin{aligned} & G(1)=C_{7}\times C_{2}, &  & G(3)=D_{7}\times C_{2}, &  & G(5)=G_{21}\times C_{2}, &  & G(7)=G_{42}\times C_{2},\\
 & G(8)=C_{7}\wr_{X_{2}}C_{2}, &  & G(20)=D_{7}\wr_{X_{2}}C_{2}, &  & G(26)=G_{21}\wr_{X_{2}}C_{2}, &  & G(45)=G_{42}\wr_{X_{2}}C_{2},\\
 & G(29)=C_{2}\wr_{X_{7}}C_{7}, &  & G(38)=C_{2}\wr_{X_{7}}D_{7}, &  & G(44)=C_{2}\wr_{X_{7}}G_{21}, &  & G(48)=C_{2}\wr_{X_{7}}G_{42}.
\end{aligned}
\]

\textbf{Class $\bm{2}$.} $G$ contains a normal subgroup $N_{7}:=\left\langle \sigma_{1}\sigma_{2}\right\rangle \simeq C_{7}$:
\[
\begin{aligned} & G(2)=N_{7}\rtimes\left\langle \tau_{1}^{3}\right\rangle \simeq D_{7}, &  & G(4)=N_{7}\rtimes\left\langle \tau_{1}\right\rangle \simeq G_{42}.\end{aligned}
\]

\textbf{Class $\bm{3}$.} $G$ contains a normal subgroup $N_{49}:=\left\langle \sigma_{1},\sigma_{2}\right\rangle \simeq C_{7}^{2}$:
\[
\begin{aligned} & G(12)=N_{49}\rtimes\left\langle \lambda_{1}\right\rangle , &  & G(13)=N_{49}\rtimes\left\langle \lambda_{2},\lambda_{3}\lambda_{4}\right\rangle ,\\
 & G(14)=N_{49}\rtimes\left\langle \lambda_{2},\lambda_{5}\lambda_{6}\right\rangle , &  & G(15)=N_{49}\rtimes\left\langle \lambda_{2},\lambda_{5}^{-1}\lambda_{6}\right\rangle ,\\
 & G(22)=N_{49}\rtimes\left\langle \lambda_{1},\lambda_{5}^{-1}\lambda_{6}\right\rangle , &  & G(23)=N_{49}\rtimes\left\langle \lambda_{1},\lambda_{5}\lambda_{6}\right\rangle ,\\
 & G(24)=N_{49}\rtimes\left\langle \lambda_{2},\lambda_{3}\lambda_{4},\lambda_{5}\lambda_{6}\right\rangle , &  & G(25)=N_{49}\rtimes\left\langle \lambda_{2},\lambda_{3}\lambda_{4},\lambda_{5}^{-1}\lambda_{6}\right\rangle ,\\
 & G(31)=N_{49}\rtimes\left\langle \lambda_{2},\lambda_{3},\lambda_{4},\lambda_{5}^{-1}\lambda_{6}\right\rangle , &  & G(32)=N_{49}\rtimes\left\langle \lambda_{2},\lambda_{3},\lambda_{4},\lambda_{5}\lambda_{6}\right\rangle ,\\
 & G(36)=N_{49}\rtimes\left\langle \lambda_{1},\lambda_{5},\lambda_{6}\right\rangle , &  & G(37)=N_{49}\rtimes\left\langle \lambda_{2},\lambda_{3}\lambda_{4},\lambda_{5},\lambda_{6}\right\rangle .
\end{aligned}
\]

\textbf{Class $\bm{4}$.} $G$ contains a normal subgroup $N_{8}:=\left\langle \mu_{0}\mu_{1}\mu_{2}\mu_{5},\mu_{0}\mu_{2}\mu_{3}\mu_{4},\mu_{0}\mu_{1}\mu_{4}\mu_{6}\right\rangle \simeq C_{2}^{3}$:
\[
\begin{aligned} & G(6)=N_{8}\rtimes\left\langle \sigma_{1}\sigma_{2}\right\rangle \simeq C_{2}^{3}\rtimes C_{7}, &  & G(11)=N_{8}\rtimes\left\langle \sigma_{1}\sigma_{2},\tau_{2}^{2}\right\rangle \simeq C_{2}^{3}\rtimes G_{21}.\end{aligned}
\]

\textbf{Class $\bm{5}$.} $G$ contains a normal subgroup $N_{16}:=\left\langle \mu_{0}\mu_{1}\mu_{2}\mu_{5},\mu_{0}\mu_{2}\mu_{3}\mu_{4},\mu_{0}\mu_{1}\mu_{4}\mu_{6},\mu_{1}\mu_{2}\mu_{4}\right\rangle \simeq C_{2}^{4}$:
\[
\begin{aligned} & G(9)=N_{16}\rtimes\left\langle \sigma_{1}\sigma_{2}\right\rangle \simeq C_{2}^{4}\rtimes C_{7}, &  & G(18)=N_{16}\rtimes\left\langle \sigma_{1}\sigma_{2},\tau_{2}^{2}\right\rangle \simeq C_{2}^{4}\rtimes G_{21}.\end{aligned}
\]

\textbf{Class $\bm{6}$.} $G$ contains a normal subgroup $N_{64}:=\left\langle \mu_{0}\mu_{1},\mu_{0}\mu_{2},\mu_{0}\mu_{3},\mu_{0}\mu_{4},\mu_{0}\mu_{5},\mu_{0}\mu_{6}\right\rangle \simeq C_{2}^{6}$:
\[
\begin{aligned} & G(21)=N_{64}\rtimes\left\langle \sigma_{1}\sigma_{2}\right\rangle \simeq C_{2}^{6}\rtimes C_{7}, &  & G(35)=N_{64}\rtimes\left\langle \sigma_{1}\sigma_{2},\tau_{2}^{2}\right\rangle \simeq C_{2}^{6}\rtimes G_{21},\\
 & G(27)=N_{64}\rtimes\left\langle \sigma_{1}\sigma_{2},\tau_{1}^{3}\right\rangle \simeq C_{2}^{6}\rtimes D_{7}, &  & G(28)=N_{64}\rtimes\left\langle \sigma_{1}\sigma_{2},\tau_{2}^{3}\right\rangle \simeq C_{2}^{6}\rtimes D_{7},\\
 & G(40)=N_{64}\rtimes\left\langle \sigma_{1}\sigma_{2},\tau_{1}\right\rangle \simeq C_{2}^{6}\rtimes G_{42}, &  & G(41)=N_{64}\rtimes\left\langle \sigma_{1}\sigma_{2},\tau_{2}\right\rangle \simeq C_{2}^{6}\rtimes G_{42}.
\end{aligned}
\]

\textbf{Class $\bm{7}$.} $G$ is the whole group $S_{14}$:
\[
\begin{aligned} & G(63)=S_{14}.\end{aligned}
\]

\textbf{Class $\bm{8}$.} The direct product and the wreath product
of smaller groups:
\[
\begin{aligned} & G(49)=S_{7}\times C_{2}, &  & G(61)=S_{7}\wr_{X_{2}}C_{2}, &  & G(57)=C_{2}\wr_{X_{7}}S_{7}.\end{aligned}
\]

\textbf{Class $\bm{9}$.} $G$ contains the normal subgroup $N_{64}$
defined in Class $6$:
\[
\begin{aligned} & G(54)=N_{64}\rtimes\left\langle \sigma_{1}\sigma_{2},\nu_{1}\right\rangle \simeq C_{2}^{6}\rtimes S_{7}, &  & G(55)=N_{64}\rtimes\left\langle \sigma_{1}\sigma_{2},\nu_{2}\right\rangle \simeq C_{2}^{6}\rtimes S_{7}.\end{aligned}
\]

\textbf{Class $\bm{10}$.} $G$ is isomorphic to $\textup{PSL}_{2}(\mathbb{F}_{7})$:
\[
\begin{aligned} & G(10)=\left\langle \sigma_{1}\sigma_{2},\nu_{3}\right\rangle \simeq\textup{PSL}_{2}(\mathbb{F}_{7}).\end{aligned}
\]

\textbf{Class $\bm{11}$.} The direct product and the wreath product
of smaller groups:
\[
\begin{aligned} & G(19)=\textup{PSL}_{2}(\mathbb{F}_{7})\times C_{2}, &  & G(52)=\textup{PSL}_{2}(\mathbb{F}_{7})\wr_{X_{2}}C_{2}, &  & G(51)=C_{2}\wr_{X_{7}}\textup{PSL}_{2}(\mathbb{F}_{7}),\\
 & G(47)=A_{7}\times C_{2}, &  & G(58)=A_{7}\wr_{X_{2}}C_{2}, &  & G(56)=C_{2}\wr_{X_{7}}A_{7}.
\end{aligned}
\]

\textbf{Class $\bm{12}$.} $G$ contains a normal elementary abelian
$2$-subgroup:
\[
\begin{aligned} & G(17)\simeq C_{2}\times\textup{PSL}_{2}(\mathbb{F}_{7}),\\
 & G(33)\text{ is an extension of }\textup{PSL}_{2}(\mathbb{F}_{7})\text{ by }C_{2}^{3}, &  & G(34)\simeq C_{2}^{3}\rtimes\textup{PSL}_{2}(\mathbb{F}_{7}),\\
 & G(42)\text{ is an extension of }\textup{PSL}_{2}(\mathbb{F}_{7})\text{ by }C_{2}^{4}, &  & G(43)\simeq C_{2}^{4}\rtimes\textup{PSL}_{2}(\mathbb{F}_{7}),\\
 & G(50)\simeq C_{2}^{6}\rtimes\textup{PSL}_{2}(\mathbb{F}_{7}), &  & G(53)\simeq C_{2}^{6}\rtimes A_{7}.
\end{aligned}
\]

\textbf{Class $\bm{13}$.} $G$ does not contain any solvable normal
subgroups:
\[
\begin{aligned} & G(16)\simeq\textup{PGL}_{2}(\mathbb{F}_{7}), &  & G(30)\simeq\textup{PSL}_{2}(\mathbb{F}_{13}), &  & G(39)\simeq\textup{PGL}_{2}(\mathbb{F}_{13}),\\
 & G(46)\simeq S_{7}, &  & G(59)\simeq A_{7}^{2}\rtimes C_{4}, &  & G(60)\simeq A_{7}^{2}\rtimes C_{2}^{2}, &  & G(62)=A_{14}.
\end{aligned}
\]

The following theorem is a strengthened version of Theorem \ref{t1.8}.
\begin{thm}
\label{t4.1} Let $k$ be a field and $G$ a transitive subgroup of
$S_{14}$. Then $k(x_{1},\cdots,x_{14})^{G}$ is k-rational provided
that
\begin{enumerate}
\item $G=G(i)$ belongs to Classes $1$, $2$, $7$, and $8$;
\item $\textup{char}(k)=7$ and $G=G(i)$ belongs to Class $3$;
\item $\textup{char}(k)\neq2$ and $G=G(i)$ belongs to Classes $4$, $5$,
$6$, and $9$;
\item $k\supseteq\mathbb{Q}(\sqrt{-7})$ and $G=G(i)$ belongs to Class
$10$.
\end{enumerate}
\end{thm}

We remark that we don't know whether $k(x_{1},\cdots,x_{14})^{G}$
is $k$-rational if $G=G(i)$ belongs to Classes $11$\textendash $13$.

\section{\label{sec5}Proof of Theorem \ref{t4.1}}

\subsection{\label{sec5.1}\texorpdfstring{$\bm{G=G(i)}$}{$G=G(i)$} belonging
to Classes \texorpdfstring{$\bm{1}$}{$1$}, \texorpdfstring{$\bm{7}$}{$7$},
and \texorpdfstring{$\bm{8}$}{$8$}}
\begin{proof}
Apply Theorem \ref{t1.1}, \ref{t2.6}, and \ref{t2.7}. For example,
consider $G(48)=C_{2}\wr_{X_{7}}G_{42}$. Note that $G_{42}$ is a
subgroup of $S_{7}$ by Definition \ref{d2.3}. We may identify $C_{2}$
with the symmetric group $S_{2}$. By Definition \ref{d2.5}, $G(48)$
may be regarded as a subgroup of $S_{14}$. Now apply Theorem \ref{t1.1}
and \ref{t2.6}.
\end{proof}

\subsection{\label{sec5.2}\texorpdfstring{$\bm{G=G(i)}$}{$G=G(i)$} belonging
to Classes \texorpdfstring{$\bm{2}$}{$2$} and \texorpdfstring{$\bm{10}$}{$10$}}
\begin{proof}
Write $\sigma:=\sigma_{1}\sigma_{2}$. Define $y_{i}:=x_{2i+1}+x_{2i+2}$
for $0\leq i\leq6$.

To show that $k(x_{1},\cdots,x_{14})^{G}$ is $k$-rational, we apply
Theorem \ref{t2.1}. Since $G$ is faithful on $k(y_{0},y_{1},\cdots,y_{6})$,
it suffices to show that $k(y_{0},y_{1},\cdots,y_{6})^{G}$ is $k$-rational.

It is routine to check that
\[
\begin{aligned} & \sigma:y_{i}\mapsto y_{i+1}, &  & \tau_{1}:y_{i}\mapsto y_{3i}, &  & \nu_{3}:y_{0}\leftrightarrow y_{2},y_{1}\leftrightarrow y_{5},y_{i}\mapsto y_{i},\text{ for }i=3,4,6,\end{aligned}
\]
where the indices of $y_{i}$ are taken modulo $7$.

The rationality of $k(y_{0},y_{1},\cdots,y_{6})^{G}$ follows from
Theorem \ref{t1.1}.
\end{proof}

\subsection{\label{sec5.3}\texorpdfstring{$\bm{G=G(i)}$}{$G=G(i)$} belonging
to Class \texorpdfstring{$\bm{3}$}{$3$}}
\begin{proof}
We rename the original variables $\{x_{1},\cdots,x_{14}\}$ by the
new names $\{x_{0},y_{0},x_{1},y_{1},\cdots,x_{6},y_{6}\}$. Then
\[
\begin{aligned} & \sigma_{1}:x_{i}\mapsto x_{i+1}, &  & \sigma_{2}:y_{i}\mapsto y_{i+1}, &  & \lambda_{1}:x_{i}\mapsto y_{-i},y_{i}\mapsto x_{i}, &  & \lambda_{2}:x_{i}\leftrightarrow y_{i},\\
 & \lambda_{3}:x_{i}\mapsto x_{-i}, &  & \lambda_{4}:y_{i}\mapsto y_{-i}, &  & \lambda_{5}:x_{i}\mapsto x_{2i}, &  & \lambda_{6}:y_{i}\mapsto y_{2i}.
\end{aligned}
\]
Note that
\[
\begin{aligned} & \lambda_{3}=\rho_{1}\text{ for }a=-1, &  & \lambda_{4}=\rho_{2}\text{ for }a=-1, &  & \lambda_{5}=\rho_{1}\text{ for }a=2, &  & \lambda_{6}=\rho_{2}\text{ for }a=2.\end{aligned}
\]
The first six cases $G=G(i)$ where $i=12,13,14,15,22,23$ are covered
by Theorem \ref{t1.7}. More specifically,
\[
\begin{cases}
G(12)\text{ corresponds the first (and the second) group for }(p,a)=(7,-1),\\
G(13)\text{ corresponds the third (and the fourth) group for }(p,a)=(7,-1),\\
G(14)\text{ corresponds the third group for }(p,a)=(7,2),\\
G(15)\text{ corresponds the fourth group for }(p,a)=(7,2),\\
G(22)\text{ corresponds the second group for }(p,a)=(7,2),\\
G(23)\text{ corresponds the first group for }(p,a)=(7,2).
\end{cases}
\]
For the remaining cases, i.e., $G=G(i)$ where $i=24,25,31,32,36,37$,
although we cannot apply Theorem \ref{t1.7} directly, the proofs
are almost the same because these groups are defined by the similar
fashion as before. Here are the indications:
\[
\begin{cases}
G(24)=\left\langle \sigma_{1},\sigma_{2},\lambda_{2},\rho_{1}\rho_{2}\text{ for }a=-1,\rho_{1}\rho_{2}\text{ for }a=2\right\rangle ,\\
G(25)=\left\langle \sigma_{1},\sigma_{2},\lambda_{2},\rho_{1}\rho_{2}\text{ for }a=-1,\rho_{1}^{-1}\rho_{2}\text{ for }a=2\right\rangle ,\\
G(31)=\left\langle \sigma_{1},\sigma_{2},\lambda_{2},\rho_{1}\text{ for }a=-1,\rho_{2}\text{ for }a=-1,\rho_{1}^{-1}\rho_{2}\text{ for }a=2\right\rangle ,\\
G(32)=\left\langle \sigma_{1},\sigma_{2},\lambda_{2},\rho_{1}\text{ for }a=-1,\rho_{2}\text{ for }a=-1,\rho_{1}\rho_{2}\text{ for }a=2\right\rangle ,\\
G(36)=\left\langle \sigma_{1},\sigma_{2},\lambda_{1},\rho_{1}\text{ for }a=2,\rho_{2}\text{ for }a=2\right\rangle ,\\
G(37)=\left\langle \sigma_{1},\sigma_{2},\lambda_{2},\rho_{1}\rho_{2}\text{ for }a=-1,\rho_{1}\text{ for }a=2,\rho_{2}\text{ for }a=2\right\rangle .
\end{cases}
\]
\end{proof}

\subsection{\label{sec5.4}\texorpdfstring{$\bm{G=G(i)}$}{$G=G(i)$} belonging
to Classes \texorpdfstring{$\bm{4}$}{$4$}, \texorpdfstring{$\bm{5}$}{$5$},
\texorpdfstring{$\bm{6}$}{$6$}, and \texorpdfstring{$\bm{9}$}{$9$}}
\begin{proof}
We will work on the field $k$ with $\textup{char}(k)\neq2$.

Write $\sigma:=\sigma_{1}\sigma_{2}$. Define $y_{i}:=x_{2i+1}-x_{2i+2}$
for $0\leq i\leq6$.

To show that $k(x_{1},\cdots,x_{14})^{G}$ is $k$-rational, we apply
Theorem \ref{t2.1} again. Since $G$ is faithful on $k(y_{0},y_{1},\cdots,y_{6})$,
it suffices to show that $k(y_{0},y_{1},\cdots,y_{6})^{G}(s)$ is
$k$-rational (note that we add an extra variable $s$).

It is routine to check that
\[
\begin{aligned} & \mu_{i}:y_{i}\mapsto-y_{i}, &  & \sigma:y_{i}\mapsto y_{i+1},\\
 & \tau_{1}:y_{i}\mapsto-y_{3i}, &  & \tau_{2}:y_{i}\mapsto y_{3i},\\
 & \nu_{1}:y_{0}\leftrightarrow-y_{1},y_{i}\mapsto-y_{i},\text{ for }2\leq i\leq6, &  & \nu_{2}:y_{0}\leftrightarrow y_{1},y_{i}\mapsto y_{i},\text{ for }2\leq i\leq6,
\end{aligned}
\]
where the indices of $y_{i}$ are taken modulo $7$.

We claim that
\begin{eqnarray*}
k(y_{0},y_{1},\cdots,y_{6})^{N_{8}} & = & k(\prod_{i=0}^{6}y_{i},y_{j}y_{j+3}y_{j+5}y_{j+6}:1\leq j\leq6),\\
k(y_{0},y_{1},\cdots,y_{6})^{N_{16}} & = & k(\prod_{i=0}^{6}y_{i}^{2},y_{j}y_{j+3}y_{j+5}y_{j+6}:1\leq j\leq6),\\
k(y_{0},y_{1},\cdots,y_{6})^{N_{64}} & = & k(\prod_{i=0}^{6}y_{i},y_{j}^{2}:1\leq j\leq6),
\end{eqnarray*}
because all of the generators on the right-hand side are fixed by
the corresponding group and the determinants of the coefficient matrices
are given by
\[
\begin{vmatrix}1 & 1 & 1 & 1 & 1 & 1 & 1\\
1 & 1 & 0 & 0 & 1 & 0 & 1\\
1 & 1 & 1 & 0 & 0 & 1 & 0\\
0 & 1 & 1 & 1 & 0 & 0 & 1\\
1 & 0 & 1 & 1 & 1 & 0 & 0\\
0 & 1 & 0 & 1 & 1 & 1 & 0\\
0 & 0 & 1 & 0 & 1 & 1 & 1
\end{vmatrix}=8,\begin{vmatrix}2 & 2 & 2 & 2 & 2 & 2 & 2\\
1 & 1 & 0 & 0 & 1 & 0 & 1\\
1 & 1 & 1 & 0 & 0 & 1 & 0\\
0 & 1 & 1 & 1 & 0 & 0 & 1\\
1 & 0 & 1 & 1 & 1 & 0 & 0\\
0 & 1 & 0 & 1 & 1 & 1 & 0\\
0 & 0 & 1 & 0 & 1 & 1 & 1
\end{vmatrix}=16,\begin{vmatrix}1 & 1 & 1 & 1 & 1 & 1 & 1\\
0 & 2 & 0 & 0 & 0 & 0 & 0\\
0 & 0 & 2 & 0 & 0 & 0 & 0\\
0 & 0 & 0 & 2 & 0 & 0 & 0\\
0 & 0 & 0 & 0 & 2 & 0 & 0\\
0 & 0 & 0 & 0 & 0 & 2 & 0\\
0 & 0 & 0 & 0 & 0 & 0 & 2
\end{vmatrix}=64.
\]
\begin{itemize}
\item $G=G(i)$ belonging to Class $4$: Define $t:=\prod_{i=0}^{6}y_{i}$
and $z_{i}:=y_{i}y_{i+3}y_{i+5}y_{i+6}$ for $0\leq i\leq6$. We find
that $t^{4}=\prod_{i=0}^{6}z_{i}$, and
\[
\begin{aligned} & k(y_{0},y_{1},\cdots,y_{6})^{G}(s)=k(t,s,z_{0},z_{1},\cdots,z_{6})^{G/N_{8}}, &  & \sigma:t\mapsto t,z_{i}\mapsto z_{i+1}, &  & \tau_{2}^{2}:t\mapsto t,z_{i}\mapsto z_{2i},\end{aligned}
\]
where the indices of $z_{i}$ are taken modulo $7$.
\item $G=G(i)$ belonging to Class $5$: Define $t:=\prod_{i=0}^{6}y_{i}^{2}$
and $z_{i}:=y_{i}y_{i+3}y_{i+5}y_{i+6}$ for $0\leq i\leq6$. We find
that $t^{2}=\prod_{i=0}^{6}z_{i}$, and
\[
\begin{aligned} & k(y_{0},y_{1},\cdots,y_{6})^{G}(s)=k(t,s,z_{0},z_{1},\cdots,z_{6})^{G/N_{16}}, &  & \sigma:t\mapsto t,z_{i}\mapsto z_{i+1}, &  & \tau_{2}^{2}:t\mapsto t,z_{i}\mapsto z_{2i},\end{aligned}
\]
where the indices of $z_{i}$ are taken modulo $7$.
\item $G=G(i)$ belonging to Classes $6$ and $9$: Define $t:=\prod_{i=0}^{6}y_{i}$
and $z_{i}:=y_{i}^{2}$ for $0\leq i\leq6$. We find that $t^{2}=\prod_{i=0}^{6}z_{i}$,
and
\[
\begin{aligned} & k(y_{0},y_{1},\cdots,y_{6})^{G}(s)=k(t,s,z_{0},z_{1},\cdots,z_{6})^{G/N_{64}}, &  & \sigma:t\mapsto t,z_{i}\mapsto z_{i+1},\\
 & \tau_{1}:t\mapsto-t,z_{i}\mapsto z_{3i}, &  & \tau_{2}:t\mapsto t,z_{i}\mapsto z_{3i},\\
 & \nu_{1}:t\mapsto-t,z_{0}\leftrightarrow z_{1},z_{i}\mapsto z_{i},\text{ for }2\leq i\leq6, &  & \nu_{2}:t\mapsto t,z_{0}\leftrightarrow z_{1},z_{i}\mapsto z_{i},\text{ for }2\leq i\leq6,
\end{aligned}
\]
where the indices of $z_{i}$ are taken modulo $7$.
\end{itemize}
The rationality of $k(y_{0},y_{1},\cdots,y_{6})^{G}(s)=k(t,s,z_{0},z_{1},\cdots,z_{6})^{G/N}$
(where $N=N_{8},N_{16},N_{64}$) follows from Theorem \ref{t1.5}.
\end{proof}

\subsection{\label{sec5.5}A New Proof of Theorem \ref{t2.4}}

In the following, we will give a new proof of Theorem \ref{t2.4}
other than that in \cite[Theorem 3.2]{KW}, because similar techniques
will be useful in many situations when $\textup{char}(k)=p>0$.
\begin{proof}
[Proof of Theorem \ref{t2.4}] Recall that $(\mathbb{Z}/p\mathbb{Z})^{\times}=\langle\bar{a}\rangle$,
$\sigma:x_{i}\mapsto x_{i+1}$, $\tau:x_{i}\mapsto x_{ai}$, $p-1=de$,
and $G_{pd}:=\left\langle \sigma,\tau^{e}\right\rangle $, where the
indices of $x_{i}$ are taken modulo $p$. We will show that $k(x_{0},x_{1},\cdots,x_{p-1})^{G_{pd}}$
is $k$-rational.

Define 
\[
u:=\sum_{i=0}^{d-1}a^{ei}\tau^{ei}\cdot x_{0}.
\]
We find that 
\[
\tau^{e}\cdot u=a^{-e}u.
\]
Define 
\[
y:=\sum_{i=0}^{p-1}\sigma^{i}\cdot u\qquad\text{ and }\qquad z:=\sum_{i=0}^{p-1}i\sigma^{i}\cdot u.
\]
Then we find that $\sigma:y\mapsto y,z\mapsto z-y$, and $\tau^{e}:y\mapsto a^{-e}y,z\mapsto a^{-2e}z$.

To show that $k(x_{0},x_{1},\cdots,x_{p-1})^{G_{pd}}$ is $k$-rational,
we apply Theorem \ref{t2.1} because $G_{pd}$ is faithful on $k(y,z)$.
It remains to prove that $k(y,z)^{G_{pd}}$ is $k$-rational. By applying
Theorem \ref{t2.2} to $k(y,z)=k(y)(z)$, the problem is reduced to
the rationality of $k(y)^{G_{pd}}$, which is the case because of
L\"uroth's theorem. Explicitly, we know that $k(y)^{G_{pd}}=k(y^{d})$
because $\sigma\cdot y=y$ and $\tau^{e}\cdot y=a^{-e}y$.
\end{proof}
For the groups $G(2)$ and $G(4)$ in Class $2$, if we assume that
$\textup{char}(k)=7$, here is another proof of the rationality by
using the idea in the above proof.

For simplicity, write $\sigma:=\sigma_{1}\sigma_{2}$ and $\tau:=\tau_{1}$,
and the field we consider is $k(x_{1},\cdots,x_{14})$.
\begin{itemize}
\item $G=G(2)$: Define $u:=x_{1}-\tau^{3}\cdot x_{1}$, $y:=\sum_{i=0}^{6}\sigma^{i}\cdot u$,
and $z:=\sum_{i=0}^{6}i\sigma^{i}\cdot u$. Then $\sigma:y\mapsto y,z\mapsto z-y$,
and $\tau^{3}:y\mapsto-y,z\mapsto z$.
\item $G=G(4)$: Define $u:=\sum_{i=0}^{5}5^{i}\tau^{i}\cdot x_{1}$, $y:=\sum_{i=0}^{6}\sigma^{i}\cdot u$,
and $z:=\sum_{i=0}^{6}i\sigma^{i}\cdot u$. Then $\sigma:y\mapsto y,z\mapsto z-y$,
and $\tau:y\mapsto3y,z\mapsto z$.
\end{itemize}
The remaining proof is the same and therefore omitted.

\addcontentsline{toc}{chapter}{References}

Hang Fu (\href{mailto:fu@ncts.ntu.edu.tw}{fu@ncts.ntu.edu.tw})\\
National Center for Theoretical Sciences, National Taiwan University,
Taipei, Taiwan

Ming-chang Kang (\href{mailto:kang@math.ntu.edu.tw}{kang@math.ntu.edu.tw})\\
Department of Mathematics, National Taiwan University, Taipei, Taiwan

Baoshan Wang (\href{mailto:bwang@buaa.edu.cn}{bwang@buaa.edu.cn})\\
School of Mathematics and System Sciences, Beihang University, Beijing,
China

\lyxaddress{Jian Zhou (\href{mailto:zhjn@math.pku.edu.cn}{zhjn@math.pku.edu.cn})\\
School of Mathematical Sciences, Peking University, Beijing, China}
\end{document}